\documentclass[11pt]{article}
\usepackage{amsmath,amsthm,amstext,amsbsy,amssymb,amsfonts,amscd}
\usepackage{graphicx}
\usepackage{color,hyperref}
\usepackage[all]{xy}

\newcommand{\mf}{\mathfrak}
\newcommand{\mc}{\mathcal}

\newcommand{\R}{\mathbf R}
\newcommand{\C}{\mathbf C}
\newcommand{\Q}{\mathbf Q}
\newcommand{\Z}{\mathbf Z}

\newcommand{\p}{\mathfrak{p}}

\newcommand{\cc}{\mathfrak{c}}

\newcommand{\ff}{\mathfrak{f}}

\newcommand{\s}{{\sigma}}

\newcommand{\ok}{\mathcal{O}_K}
\newcommand{\oo}{\mathcal{O}}

\newcommand{\qm}{\Q_{(m)}}

\newcommand{\kf}{K_\ff}

\newcommand{\sm}[4]{{
  \bigl[\begin{smallmatrix}{#1}&{#2}\\{#3}&{#4}\end{smallmatrix}\bigr]}}
\newcommand{\sv}[2]{
  \genfrac[]{0pt}{1}{#1}{#2}}

\newcommand{\Gal}{\text{Gal}}

\newcommand{\gl}{\text{GL}}

\numberwithin{equation}{section}

\theoremstyle{plain}
\newtheorem{theorem}{Theorem}[section]

\newtheorem{proposition}[theorem]{Proposition}
\newtheorem{corollary}[theorem]{Corollary}
\newtheorem{theorem*}{Theorem}
\newtheorem{corollary*}{Corollary}

\theoremstyle{definition}

\newtheorem{example}[theorem]{Example}
\newtheorem{conjecture}[theorem]{Conjecture}
\newtheorem{remark}[theorem]{Remark}

\begin{document}

\title{Certain CM class fields with smaller generators}

\author{
\"Omer K\"u\c{c}\"uksakall{\i}\footnote{Middle East Technical University,
Department of Mathematics, 06800, Ankara, Turkey (komer@metu.edu.tr).}\ ,
Osmanbey Uzunkol\footnote{Carl von Ossietzky Universit\"at Oldenburg, Institut
f\"ur Mathematik, D-26111, Oldenburg, Germany (osmanbey.uzunkol@gmail.com).}
\footnote{Both authors are partially supported by a joint research project
funded by BMBF (01DL12038) and T\"ubitak (TBAG-112T011).}}

\date{\today}

\maketitle

\begin{abstract}

We introduce an algorithm that computes explicit class fields of an imaginary
quadratic field $K$ for a given modulus $\mathfrak{f}\subset\ok$ more
efficiently than the use of their classical counterparts. Therein, we prove the
fact that certain values of a simple quotient of Siegel $\phi$-function are
elements in the ray class field $K_{\mathfrak{f}}$ of $K$.
\end{abstract}

\section{Introduction}

Inspired by Kronecker-Weber theorem \cite{hilbert}, Hilbert's 12th problem
asks to generate the maximal abelian extension of a given number
field explicitly using singular values of an analytical function. This problem
can be regarded as finding a generalization of the exponential function
appearing in Kronecker-Weber theorem. In the case of imaginary
quadratic number fields $K$, Hilbert's 12. problem, also known as Kronecker's
Jugendtraum, has an affirmative answer. The first proof of this fact was given
by Hasse in 1927 in \cite{hasse}, and this was significantly simplified by
Deuring in 1958 in \cite{Deu58} using the theory of complex multiplication
(CM-Theory) of elliptic curves and their $j$-invariants. More precisely, as a
preliminary step one needs to construct the maximal unramified abelian extension
(Hilbert class field $H_{K}$) of the imaginary quadratic number field $K$ using
a suitable value of $j$-function. All other (ramified) class fields of $K$ are
constructed by adjoining suitable torsion values of Weber functions to the
Hilbert class field $H_{K}$, see for instance \cite{Sil94} for a rather modern
treatment. The preliminary step is not necessary in the case of rational
field $\Q$ due to the triviality of its class group.

Shimura's reciprocity law connects class field theory of imaginary quadratic
number fields and arithmetic of modular functions by means of bringing Artin's
reciprocity law and Galois theory of the field $\mathcal{F}$ of all arithmetical
modular functions together, see \cite[Chapter 6]{Sh71}. It enables us to find an
explicit action of absolute abelian Galois group of an imaginary quadratic field
on the singular values $g(\tau)$, where $g\in \mathcal{F}$ and $\tau$ is an
imaginary quadratic number lying in complex upper half plane $\mf{h}$. However,
computations are rather involved due to the presence of roots of unity in the
functions. 

CM-theory of elliptic curves plays also a vital role in construction of
elliptic curves (CM-construction) with prescribed properties, such as known
number of rational points, over finite fields. This CM-construction can be
realized much better by using suitable generators of Hilbert class fields (or
more generally ring class fields) having minimal polynomials with much smaller
coefficients than coefficients of the minimal polynomials of the values of
$j-$function used in the classical construction, see \cite{Gee01, EngMor09,
LePoUz09, Uz13} for more details. We refer to \cite{AtMr93} and \cite{Mor07} for
applications of CM-construction in primality proving, or \cite{BSS99},
\cite{BSS05} and \cite{FST06} for applications in group and pairing based
cryptography. 
 
Let $m\not\equiv 2\bmod 4$ be a natural number, and $a$ be an integer such that
$2\leq a\leq m-1$. Let further $\qm$ be the corresponding ray class field over
$\Q$. As is well known $\qm$ is the $m$-th real cyclotomic field $\Q(\zeta_m +
\zeta_m^{-1})$ and can be obtained by adjoining the values of $\psi(z) =
1-e^{2\pi i z}$ at rational numbers $a/m$ to $\Q$. Moreover, real cyclotomic
units are given by a nice expression, see \cite[p. 144]{Wa97}:

\begin{equation*}\label{eqn:cyclo}
 \zeta_m^{(1-a)/2}\frac{1-\zeta_m^a}{1-\zeta_m}=\pm\frac{\sin(\pi
a/m)}{\sin(\pi/m)} \in \oo_{\Q_{(m)}}^*.
\end{equation*}

There are analogues elliptic units in the in the case of imaginary quadratic
number fields. Our aim in this paper is to introduce an algorithm to compute
explicit class fields generated by these special units. Furthermore, we prove
that these units yield 'smaller' generators for certain class fields than their
classical counterparts. In the realm of Hilbert's 12. Problem, it is indeed the
raison d'\^etre of this paper to provide smaller primitive elements for class
fields in the ray class field $K_{\mathfrak{f}}$ over $K$, without any
restriction on the modulus $\mathfrak{f}$ of $K$, as special values of a single
analytical function, namely certain values of a simple quotient of Siegel
$\phi$-function.

By using Kronecker's limit formula, Ramachandra introduced in 1964
primitive generators given as a product of values of various analytical
functions, see \cite{Ram64}. This ray class invariants are not suitable for
explicit computations due to the presence of a very large product. A
variant of Schertz's conjecture (\cite[p. 386]{Sch97}), recently proven
in \cite[Remark~3.7, p.~424]{JKS11}, states that
\begin{equation}\label{eqn:jung}
 \phi(0, 1/N, \tau)^{12N/\gcd(6,N)}
\end{equation}
generates $K_{(N)}$ over $K$ if $K\neq \Q(\sqrt{-1}), \Q(\sqrt{-3})$ and $N\geq
2$. Under rather restrictive conditions, Bettner and Schertz showed in
\cite{BeSch01} that the expression
\begin{equation}\label{eqn:bettner}
\zeta\Theta=\zeta\prod_{i=1}^{s}\phi(u_i,v_i,\tau_i)^{n_i},
\end{equation}
with specially chosen $u_i,v_i\in\R$, $\tau_i\in\mf{h}$, $n_i\in\Z$ and a
suitable root of unity $\zeta$ turns out to be an element of $K_{(N)}$ in some
cases. They also conjectured that these elements are generators for $K_{(N)}$
over $K$ if $\zeta\Theta\in\kf$. 

One important advantage of ray class invariants introduced in this paper is that
a factor of $r(N):=12N/\gcd(6,N)$ can be gained by using a suitable simple
quotient of Siegel $\phi$-function heuristically in comparison to the
generators in
(\ref{eqn:jung}). We have also smaller elements of $K_{\mf{f}}$ compared to the
elements in the expression (\ref{eqn:bettner}). Another main advantage is that
we have no restriction on the ideal $\mathfrak{f}$ of $\mathcal{O}_K$ in
comparison to both cases (\ref{eqn:jung}) and (\ref{eqn:bettner}). In
particular, we can pick a conductor $\ff$ that is not necessarily generated by a
natural number $N$ unlike the expression (\ref{eqn:jung}).

\section{Shimura's Reciprocity Law}

Let $K$ be an imaginary quadratic number field of discriminant $d_K$, and
$\mathcal{O}_K$ be its ring of integers. We formulate the results in this paper
for the case of complex multiplication by the maximal order $\mathcal{O}_K$. For
simplicity we denote the maximal order $\mathcal{O}_K$ by $\mathcal{O}$ for a
fixed $K$.   
  
In this section, we present  an explicit version of Shimura's
reciprocity law introduced by Gee and Stevenhagen, see \cite{Gee01}
and \cite{Stev01}, respectively. Unless otherwise stated or proved, the
assertions of this section can be found in \cite{lang}, \cite{stark} and 
\cite{Sh71}. Moreover, we introduce the definition and transformation
formulas of Siegel $\phi$-function. For a detailed treatment, we refer to the
Stark's paper \cite{stark}. 

\subsection{Class Field Theory}

Let $\mbox{Cl}(K)$ be the ideal class group of $K$ and $H$ be the Hilbert class
field of $K$. We denote by $[\cdot, K]$ the Artin map on the group of finite
$K-$id\`eles $\widehat{K^{*}} = \left(\prod_{p}^{'}K \otimes_{\Q}
\mathcal{O}_{p} \right)^* = (K\otimes_ { \Z } \widehat { \Z } )^*$, i.e.
id\`eles as quotient of the full id\`ele group obtained by forgetting the
infinite component $\C^*$. In this case, we can summarize the main theorem of
class field theory simply by the following exact sequence (see \cite[p.
115]{Sh71} for the general case and \cite[p. 165]{Stev01} for imaginary
quadratic number fields):
\begin{equation}\label{eqn:CFTim}
\begin{xy}
\xymatrix{
 1\ar[r] &  K^{*}\ar[r] & \widehat{K^{*}} \ar[r]^(0.3){[\cdot, K]} &
\mbox{Gal}(K^{ab}/K) \ar[r] & 1, }
\end{xy}
\end{equation}
where $K^{ab}$ denotes the maximal abelian extension of $K$. The unit group
$$\widehat{\mathcal{O}^{*}} = \left(\lim_{\leftarrow N}(\oo/N\oo)\right)^* =
(\oo\otimes_{\Z}\widehat{\Z})^* \subseteq \widehat{K^{*}}$$ of the profinite
completion $$\widehat{\mathcal{O}} = \lim_{\leftarrow N}(\oo/N\oo) =
\mathcal{O}\otimes_{\Z}\widehat{\Z}$$ of the maximal order $\mathcal{O}$ inside
$\widehat{K^{*}}$ is the preimage of the Artin map of $\mbox{Gal}(K^{ab}/H)$
(see \cite[p. 165]{Stev01}). This implies that we obtain the following exact
sequence
by class field theory: 
\begin{equation}\label{eqn:RingclassExact}
\begin{xy}
\xymatrix{
 1\ar[r] &  \mathcal{O}^{*}\ar[r] & \widehat{\mathcal{O}^{*}}
\ar[r]^(0.3){[\cdot, K]} & \mbox{Gal}(K^{ab}/H) \ar[r] & 1.}
\end{xy}
\end{equation}

\subsection{Modular Functions}

Let $\tau$ be an element in $\mf{h}\cap K$ having minimal polynomial
$Ax^2+Bx+C$ such that $B^2-4AC=d_K$. The $j$-function is invariant under
$\Gamma:=\mbox{SL}(2,\Z)$, and the main theorem of complex multiplication
implies that $j(\tau)$ generates the Hilbert class field $H$ over $K$.  

In order to generate other (ramified) abelian extensions of $K$, one can use
modular functions.  A modular function of level $N$ is defined as a
meromorphic function on $\mf{h}$, which is invariant under the congruence
subgroup $\Gamma(N)=\mbox{ker}[\mbox{SL}(2,\Z)\rightarrow \mbox{SL}(2,\Z/N\Z)]$
of $\Gamma$. A modular function of level $N$, whose $q$-expansions at every cusp
have coefficients in $\Q(\zeta_N)$, is called arithmetical. The field of all
arithmetical modular functions of level $N$ is abbreviated by $\mathcal{F}_{N}$.
In particular, $\mathcal{F}_1=\Q(j).$ From now on we use the term modular
function instead of arithmetical modular function for simplicity.
\begin{remark}
 As a consequence of the main theorem of complex multiplication, we have the
property that for every modular function $g\in\mathcal{F}_N$, the value
$g(\tau)$, if finite, is contained in the ray class field $K_{(N)}$
of $K$ with conductor $(N)$, see for instance \cite[p. 41]{Gee01}.
\end{remark}

One can show that $\mc{F}_N$ is a Galois extension of $\mc{F}_1$ with
\begin{align*}
 \Gal(\mc{F}_N/\mc{F}_1) & \cong \mbox{SL}(2,\Z/N\Z)/\{\pm
I_{2}\}\rtimes(\mathbb{Z}/N\mathbb{Z})^{*} \\
&\cong \gl_2(\Z/N\Z)/\pm I_2,
\end{align*}
where $I_2$ denotes the $2\times 2$ identity matrix. The action of this Galois
group on modular
functions $\mc{F}_N$ can be described easily. If $A\in\Gamma$, then we have
\[f(z) \circ A = f(Az),\]
and if $A = \sm{1}{0}{0}{d}$ for $d\in\Z$ with $(d,N)=1$, then
\[ f(z) \circ A =  \left( \sum_{n=n_0}^\infty \alpha_n q_N^n \right) \circ A =
\sum_{n=n_0}^\infty \alpha_n^{\s} q_N^n, \]
where $\s$ is the automorphism of $\Q(\zeta_N)/\Q$ given by $\zeta_N^\s =
\zeta_N^d$.

In order to describe the id\`elic interpretation of modular
functions of all levels, Gee considers the following diagram of exact sequences
\cite[p. 10]{Gee01}:
\begin{equation}\label{eqn:galFN1}
\begin{xy}
\xymatrix{
1 \ar[r] & \{\pm 1\} \ar[d] \ar[r] & \mbox{SL}(2,\Z/N\Z) \ar[d] \ar[r] &
\mbox{Gal}(\mathcal{F}_{N}/\mathcal{F}_{1}(\zeta_{N})) \ar[d] \ar[r] & 1
\\
1  \ar[r] & \{\pm 1\} \ar[d] \ar[r] & \mbox{GL}(2,\Z/N\Z) \ar[d]^{\det} \ar[r] &
\mbox{Gal}(\mathcal{F}_{N}/\mathcal{F}_{1})\ar[d] \ar[r] & 1 \\
1 \ar[r] & 1 \ar[r] & \left(\Z/N\Z\right)^{*} \ar[r] & \mbox{Gal}(\mathcal{F}_{1}(\zeta_{N})/\mathcal{F}_{1}) \ar[r] & 1.  
}
\end{xy}
\end{equation}
Let $\mathcal{F}=\cup_{N\geq 1}\mathcal{F}_{N}$ be the field of all 
modular functions. We can describe the Galois group of $\mathcal{F}$ over
$\mathcal{F}_{1}$ simply by taking the projective limit of (\ref{eqn:galFN1}):
 \begin{equation}\label{eqn:galFN}
\begin{xy}
\xymatrix{
 1\ar[r] &  \{\pm 1\}\ar[r] & \mbox{GL}(2,\mathbb{\widehat{Z}}) \ar[r] &
\mbox{Gal}(\mathcal{F}/\mathcal{F}_{1}) \ar[r] & 1. }
\end{xy}
\end{equation}

\subsection{Reciprocity Law}

We follow the explicit version of Shimura's reciprocity law due to Stevenhagen,
see \cite{Stev01}. Reciprocity law of Shimura connects the exact sequences
(\ref{eqn:RingclassExact}) and (\ref{eqn:galFN}) with the following reciprocity
map
$h_\tau$:
\begin{equation}\label{eqn:rec}
\begin{xy}
\xymatrix{
1\ar[r] &  \mathcal{O}^{*} \ar[r] & \prod_{p}^{'}\mathcal{O}^{*}_{p} \ar[r]
\ar[d]^{h_{\tau}} & \mbox{Gal}(K^{ab}/H)\ar[r] & 1\\
1\ar[r] &  \{\pm 1\}\ar[r] & \mbox{GL}(2,\mathbb{\widehat{Z}}) \ar[r] & \mbox{Gal}(\mathcal{F}/\mathcal{F}_{1}) \ar[r] & 1,
}
\end{xy}
\end{equation}
where $h_\tau: \prod_{p}^{'}\mathcal{O}^{*}_{p} \rightarrow
\mbox{GL}(2,\mathbb{\widehat{Z}})$ sends the id\`ele $x\in
\prod_{p}^{'}\mathcal{O}^{*}_{p}$ to the transpose of the matrix representing
the
multiplication on $ \mathbb{\widehat{Z}}\cdot \tau +  \mathbb{\widehat{Z}}$ with
respect to the basis $[\tau, 1]$ when viewed as a free $
\mathbb{\widehat{Z}}-$module of rank 2. We have the following
explicit formula for the reciprocity map in (\ref{eqn:rec}):
\begin{equation}\label{eqn:recexp} 
h_\tau:\ x=sA\tau+t\mapsto \sm{t-Bs}{-Cs}{sA}{t}.
\end{equation}
Using the reciprocity map, we obtain an action of $\widehat{\mathcal{O}^{*}}$ on
the full arithmetic modular function field $\mathcal{F}$, \cite[p. 165]{Stev01},
via
\begin{equation*}\label{eqn:rec2}
 (g(\tau))^{[x^{-1}, K]}=(g^{h_\tau(x)})(\tau).
\end{equation*}
Stevenhagen reduces the the reciprocity law of Shimura to the exact sequences of
finite groups, \cite[p. 167]{Stev01}
\begin{equation*}\label{eqn:rec3}
 \begin{xy}
\xymatrix{
1 \ar[r] &\mathcal{O}^{*} \ar[r] & (\mathcal{O}/N\mathcal{O})^{*} \ar[r]
\ar[d]^{h_{\tau, N}}& \mbox{Gal}(K_{(N)}/H)\ar[r] & 1\\
1 \ar[r] &\{\pm 1\}\ar[r] & \mbox{GL}(2,\mathbb{Z}/N\mathbb{Z}) \ar[r] &
\mbox{Gal}(\mathcal{F}_{N}/\Q(j)) \ar[r] & 1.
}
\end{xy}
\end{equation*}
The image of $(\mathcal{O}/N\mathcal{O})^{*}$ under the reciprocity map
$h_{\tau,N}$  is the following subgroup of
$\mbox{GL}(2,\mathbb{Z}/N\mathbb{Z})$:

 \begin{equation}\label{eqn:imrec}
 \mathcal{W}_{N,\tau}=\left\{\sm{t-Bs}{-Cs}{sA}{t}\in
\mbox{GL}(2,\mathbb{Z}/N\mathbb{Z}):\ s,t\in \mathbb{Z}/N\mathbb{Z}\right\}.
\end{equation}

Let $\mbox{Cl}(d_K)$ be the form class group, i.e. the group consisting of
reduced binary quadratic forms. It is a well known fact 
that the form class group $\mbox{Cl}(d_K)$ and the ideal class
groups $\mbox{Cl}(K)$ are
isomorphic, see \cite[p. 50]{cox}. This isomorphism is given by mapping a
reduced binary quadratic form $Q=[a,b,c]$ to the ideal class
containing the fractional ideal generated by $\tau_Q=(-b+\sqrt{dk})/2a$ and $1$.
Gee \cite[Chapter 1]{Gee01} proved the following theorem:

\begin{theorem}\label{thm:gee}
Let $Q=[a,b,c]$ be a reduced binary quadratic form of discriminant $d_K$ and
$\tau_Q=(-b+\sqrt{dk})/2a$. Set $u_Q=(u_p)_p\in\prod_p\emph{GL}(2,\Z_p)$ with
\begin{itemize}
 \item For $d_K\equiv 0\bmod 4:$\\
$u_p=
\begin{cases}
 \sm{a\ }{b/2}{0}{1} & \text{ if } p\not|a,\\
 \sm{-b/2\ }{-c}{1}{0} & \text{ if }p|a \text{ and }p\not|c,\\
  \sm{-a-b/2\ }{-c-b/2}{1}{-1} & \text{ if } p|a \text{ and }p|c,\\
\end{cases}
$

 \item For $d_K\equiv 1\bmod 4:$\\
$u_p=
\begin{cases}
 \sm{a\ }{(b-1)/2}{0}{1} & \text{ if } p\not|a,\\
 \sm{-(b+1)/2\ }{-c}{1}{0} & \text{ if }p|a \text{ and }p\not|c,\\
  \sm{-a-(b+1)/2\ }{-c-(1-b)/2}{1}{-1} & \text{ if } p|a \text{ and }p|c.\\
\end{cases}
$
\end{itemize}
Let $g\in\mathcal{F}$ be a modular function. Then it holds
$$g(\tau)^{[x_Q^{-1}, K]}=g^{u_Q}(\tau_Q),$$
where $g$ is defined and finite at $\tau$, and $x_Q=(x_p)_p$ with\\
$$x_p=
\begin{cases}
 a & \text{ if } p\not|a,\\
 a\tau_Q& \text{ if }p|a \text{ and }p\not|c,\\
  a(\tau_Q-1) & \text{ if } p|a \text{ and }p|c.\\
\end{cases}
$$

\end{theorem}

In \cite[p. 418--420]{JKS11} the following theorem is proven:
\begin{theorem}\label{thm:jung}
 Assume that $K\neq\Q(\sqrt{-1}),\ \Q(\sqrt{-3})$ and $N>0$. Then there exists a
bijective map $\Psi$ with $$\Psi:\mathcal{W}_{N,\tau}/\{\pm I_2\}\times
\emph{Cl}(d_K)\longrightarrow \emph{Gal}(K_{(N)}/K)$$
where 
$$(\alpha, Q)\longmapsto (g(\tau)\mapsto g^{\alpha\cdot
u_Q}(\tau_Q))_{g\in \mathcal{F}_{N,\tau}},$$
and $\mathcal{F}_{N,\tau}$ is the set of modular functions
of level $N$, which are defined and finite at $\tau$. 
\end{theorem}

\subsection{Transformation formulas for Siegel phi-function}

Stark obtains elements in the ray class fields $K_{\ff}$ by evaluating
the modular 
function $\phi(u,v,z)$, Siegel $\phi$-function, at imaginary
quadratic numbers. 
Set $\gamma=uz+v$, for $u,v\in\R$ and $z\in\mf{h}$. Siegel $\phi$-function is defined by the infinite product
\begin{equation*}\label{phi}
\phi(u,v,z) = -i e^{\frac{\pi i z}{6}}e^{\pi i u \gamma}(e^{\pi i \gamma} -
e^{-\pi i \gamma}) \prod_{n=1}^{\infty} (1-e^{2\pi i(nz+\gamma)}) 
(1-e^{2\pi i(nz-\gamma)}).
\end{equation*}
\begin{proposition}\label{phiprop}
The function $\phi(u,v,z)$ satisfies the following transformation properties:
\begin{enumerate}
 \item $\phi(u,v+1,z) = -e^{\pi i u}\phi(u,v,z)$
 \item $\phi(u+1,v,z) = -e^{-\pi i v}\phi(u,v,z) $
 \item $\phi(u,v,z+1) = e^{\pi i/6} \phi(u,u+v,z)  $
 \item $\phi(u,v,-1/z) = e^{-\pi i/2} \phi(v,-u,z)  $
\end{enumerate}
\end{proposition}
\begin{proof}
This is a consequence of Kronecker's second limit formula, see for example
\cite[p.~207-208]{stark} for details.
\end{proof}

Let $N>1, s, t$ be integers with $(s,t,N)=1$. Suppose
that $u=s/N, v=t/N$ and $M=12N^2$. Then Siegel $\phi$-function
$\phi\left(u, v, z\right)$ is a modular function of level $M$, see \cite[p.
208]{stark}. We now consider the action of some basic matrices in
$\gl(2,\Z/M\Z)$ on modular functions $\phi(u,v,z)$. We first start with
$\text{SL}(2,\Z)$ which is generated by elements $T=\sm1101$ and $S=\sm0{-1}10$.
Define the multiplicative homomorphism $\omega: \text{SL}_2(\Z) \rightarrow
\langle\zeta_{12}\rangle$ which maps $T \mapsto e^{\pi i/6}$ and $S \mapsto
e^{-\pi i/2}$. This map is compatible with Proposition~\ref{phiprop}, and it
follows
\[ \phi(u,v,z) \circ A = \phi(u,v,Az) = \omega(A) \phi((u,v)A,z) \]
for an integral matrix $A$ with $\det(A)=1$. In particular,
$S^2=\sm{-1}{0}{0}{-1}$, and we have
\[\phi(u,v,z) = -\phi(-u,-v,z).\]
In general, we do not have to
decompose $A$ in terms of $S$ and $T$ to be able to compute the value of
$\omega(A)$. Given $A=\sm abcd\in\Gamma$, define
\begin{align*}
p_3(A)&=ac(b^2+1)+bd(a^2+1) \\
p_4(A)&=(b^2-a+2)c+(a^2-b+2)d+ad.
\end{align*}
Herglotz \cite{herg} gives the following formula:
\begin{equation}\label{herg}
 \omega(A) = \zeta_4^{p_4(A)} \zeta_3^{-p_3(A)}.
\end{equation}

At the next step, we compute the action of $\sm{1}{0}{0}{d}$ on $\phi(u,v,z)$,
where $d$ is an integer relatively prime to $M$. The function $\phi(u,v,z)$ has
coefficients in $\Q(\zeta_M)$, and the action of the automorphism $\s: \zeta_M
\mapsto \zeta_M^d$ on $\phi(u,v,z)$ is obtained by multiplying $v$ by $d$ except
$-i$ at the beginning. Note that $\s(-i) = -i(-1)^{(d-1)/2}$. Therefore 
\[ \phi(u,v,z) \circ \sm{1}{0}{0}{d} = \phi(u,vd,z)(-1)^{(d-1)/2}. \]
The action of $\sm{d}{0}{0}{1}= \sm{0}{-1}{1}{0} \sm{1}{0}{0}{d}
\sm{0}{-1}{1}{0}$ in $\gl(2,\Z/M\Z)$ can also be easily computed:
\begin{align}\notag
 \phi(u,v,z)\circ \sm{d}{0}{0}{1} & = \phi(u,v,z) \circ \sm{0}{-1}{1}{0}
\sm{1}{0}{0}{p} \sm{0}{-1}{1}{0}\\ \notag
& = -i \phi(v,-u,z) \circ \sm{1}{0}{0}{d} \sm{0}{-1}{1}{0}\\ \notag
& = -i\phi(v,-ud,z) \circ \sm{0}{-1}{1}{0}\\ \notag
& = -\phi(-ud,-v,z)\\ 
\label{rouact} & = \phi(ud,v,z). 
\end{align}

\section{Elliptic Units}

Let $K$ be an imaginary quadratic field and let $\ff \subset \ok$ be a proper
ideal. In this section we give an algorithm to compute a complete set of
conjugates of a suitable simple quotient of values of Siegel $\phi$-function
over $K$. These special values turn out to be elliptic units in $\kf$. The
inspiration comes from Stark's basic result for $\ff=\p^s$ where $\p$ is a
degree one prime ideal coprime to $6d_K$ and $s\in\Z^{>0}$.
\cite[p.~229]{stark}. 

Suppose $f$ is the minimal positive integer divisible by the ideal $\ff\neq
(1)$. For each ideal class $\cc$ in $\mbox{Cl}_\ff=I_K(\ff) / P_{K,1}(\ff)$,
choose any ideal
$\mf{b}$ coprime to $\ff$ such that $\mf{a}\mf{b}$ is principal for all ideals
$\mf{a}$
in $\cc$. Furthermore, we assume that $\mf{a}\mf{b} = (\alpha)$ for some
$\alpha\in\oo$ and
$\mf{b}\ff = [\omega_1,\omega_2]$ where $\tau = \omega_1/\omega_2\in\mf{h}$. We
write $ \alpha = [u,v]\sv{\omega_1}{\omega_2}$, where $u,v$ are rational
numbers such that $fu$ and $fv$ are integers. The elements 
\begin{equation}\label{eqn:eunit}
 E(\cc) = \phi( u,v,\tau)^{12f} \in \kf
\end{equation}
depend only upon $\cc$ by \cite[Lemma 7]{stark}. The argument $\tau$ corresponds
to a fractional ideal, and it can be transformed to $\tau_Q$ for some reduced
binary quadratic form $Q$. Without loss of generality, we can always pick
$\tau=\tau_Q$.  

Let $\p \subset \ok$ be a degree one prime ideal in the ideal class
$\cc$ of norm $|\p|=p$ with $(p,12f)=1$. Stark shows that $E(\cc)/E(1)^p$
is a $12f$-th power of a number in $\kf$. Moreover, if $\kf$ contains exactly
$W$ roots of unity, then $W|12f$ and $E(\cc)^W$ is a $12f$-th power of an
algebraic integer in $\kf$ \cite[Lemma~9]{stark}. Let $\sigma_{\cc}$ be the
element of $\mbox{Gal}(\kf/K)$ corresponding the ideal class $\cc$ under the
isomorphism $\mbox{Gal}(\kf/K)\cong \mbox{Cl}_\ff$. The action of
$\sigma_{\cc}$ on the roots of unity is given by
$\zeta_W^{\sigma_{\cc}} = \zeta_W^{d_{\cc}}$
where $d_\cc \equiv |p| \bmod{W}$. We define $e_{\cc} = W/(W,d_{\cc}-1)$.
Inspired by Stark, we find an element 
\begin{equation}\label{eqn:oran}
\left(\frac{E(\cc)}{E(1)}\right)^{e_{\cc}}
\end{equation}
of norm $1$ which is a $12f$-th power of an algebraic integer in $\kf$
by \cite[Theorem~1]{stark}. In particular, this process provides us
examples of elliptic units. 

We can choose an ideal class $\cc_0$ modulo $\ff$ such that its restriction to
the Hilbert class field $H$ contains the ideal $\ff$. For any choice of
$\mf{b}$, we have the property that $\mf{b}\ff$ is principle if $\mf{a}\mf{b}$
is principle. It follows that we can choose $u_1,v_1\in (1/f)\Z$ so that 
\begin{equation}\label{eqn:tau1}
  E(\cc_0)=\phi(u_1,v_1,\tau_1)^{12f}.
\end{equation}

If $\ff = (N)$, an ideal generated by a positive integer $N$, then we can choose
$\mf{b}=1$ in the above setting and start with the pair $[u_1,v_1]=[0,1/N]$ in
the expression (\ref{eqn:tau1}). Otherwise $[u_1,v_1]$ can be found by using
a suitable ideal $\mf{b}$ coprime to $\ff$. 

By Stark's version of reciprocity law \cite[p.~223]{stark}, we have 
\[ \sigma_{\cc}(E(\cc_0)) = E(\cc\cc_0) = \phi(u_\cc,v_\cc,\tau_\cc)^{12f}\]
for some $u_\cc,v_\cc\in (1/f)\Z$ and $\tau_\cc=\tau_Q$ for some $Q$. In this
case the restriction of $\cc$ to the Hilbert class field and the reduced
binary quadratic form $Q$ correspond to the same ideal class.

From now on we will focus on the case $\tau_\cc = \tau_1$ for simplicity.
Moreover, we will assume $e_\cc=1$, in order to construct minimal polynomials
with coefficients as small as possible. We refer to Example~\ref{nontrivial}
for a possible comparison of several cases. The conditions $\tau_\cc=\tau_1$
and $e_\cc=1$ hold if and only if $\cc$ is a class of principal ideals such that
$\sigma_\cc$ acts trivially on the roots of unity of $\kf$. Such a class
$\cc\neq 1$ exists if and only if $\kf\neq H(\zeta_W)$ by class field theory.

\begin{remark}
If the equality $\kf = H(\zeta_W)$ holds, then one can easily generate the ray
class field $\kf$ using class invariants of $H$ together with cyclotomic
elements. Hence, from this point of view the condition $\kf\neq H(\zeta_W)$ is
not a restriction.
\end{remark}

\subsection{Explicit Conjugates}

Let $\cc$ be an ideal class modulo $\ff$ such that $\tau_\cc=\tau_1$ and
$e_\cc=1$.
We want to compute $[u_\cc, v_\cc]$ for a given $[u_1,v_1]$. Recall that
$W|12f$. Set $\ell = 12f/W$. Since $E(1)^W$ is a $12f$-the power of an element
in $\kf$, we find that 
\[ G(1):=\sqrt[\ell]{E(1)} = \phi(u_1,v_1,\tau_1)^W \cdot \zeta_{\ell}^* \in
\kf\]
for some $\ell$-th root of unity. Let $M=12f^2$. We can find an element
$\alpha\in\mc{W}_{M,\tau_1}$ whose restriction to $\kf$ corresponds to
the automorphism $\sigma_\cc$, and acts trivially on the extension
$\kf(\zeta_\ell)/\kf$.  Let $\widehat{\alpha}=\alpha-1$ be
the element in the group ring $\Z[G]$, where $G=\mc{W}_{M,\tau_1}/\{\pm I_2\}$.
Then, it follows that:
\[G(1)^{\widehat{\alpha}} = \left[
\frac{\phi(u_\cc,v_\cc,\tau_\cc)}{\phi(u_1,v_1,\tau_1)} \right]^W \in
\kf, \]
where $[u_\cc,v_\cc] = [u_1,v_1]\cdot\alpha$ and $\tau_\cc=\tau_1$ hold.
Furthermore, the condition $e_\cc=1$ implies that the above quotient is a $W$-th
power of an element in $\kf$ due to the equation (\ref{eqn:oran}). Since $\kf$
contains $W$-th roots of unity, we can fix that 
\[ \epsilon(\cc) := \frac{\phi(u_\cc,v_\cc,\tau_1)} {\phi(u_1,v_1,\tau_1)} \in
\kf, \]
which is well defined up to a $W$-th root of unity.

\begin{remark}
 At this point, we could also proceed with Stark's version of reciprocity law 
and obtain the same action, see for instance \cite{sakalli}. However id\`elic
interpretation introduced in the first section provides us with a better
treatment of the
subject. 
\end{remark}

Our purpose is now to compute $\epsilon(\cc)^{\sigma}$ for all $\sigma\in
\Gal(\kf/K)$. Find a pair $(\beta,Q) \in \mc{W}_{M,\tau_1} \times
\mbox{Cl}(d_K)$ such that whose restriction to $\kf$ corresponds to $\sigma$. By
Chinese remainder theorem one can compute the matrix $u_Q\bmod{M}$, see
\cite[p.~46]{Gee01}. By Theorem \ref{thm:gee} and Theorem \ref{thm:jung}, the
action of $\beta \cdot u_Q \in \mbox{GL}(2,\Z/M\Z)$ on Siegel $\phi$-function is
given explicitly. We have
\begin{equation*}\label{eqn:act1} \phi(u,v,\tau_1)^{(\beta,Q)} = \phi([u,v]\cdot
\beta \cdot
u_Q,\tau_Q) \cdot \zeta_Q 
\end{equation*}
for some $12$-th root of unity $\zeta_Q$ depending on $Q$ but not on $u,v$ and
$\beta$. The assumption $\tau_\cc = \tau_1 $ enables us to cancel $\zeta_Q$,
and simplifies computations considerably. As a result we have the following
formula
\begin{equation}\label{eqn:act2}\epsilon(\cc)^{\sigma} =
\frac{\phi([u_\cc,v_\cc] \cdot \beta \cdot u_Q,
\tau_Q)}{\phi([u_1,v_1] \cdot \beta \cdot u_Q, \tau_Q) }. 
\end{equation}
In the case that the assumption $\tau_\cc = \tau_1 $ does not hold, one
can still
perform similar computations. For a possible comparison, we refer to
Example~\ref{nontrivial}. 

We obtain the following algorithm which can be obtained immediately by using
the explicit action given in (\ref{eqn:act2}). 

 \vspace{0.3cm}
\noindent
\rule{\hsize}{0.5pt}
\textbf{Algorithm I: Computation of conjugates for $\epsilon(\cc)$}\\*[-1ex]
\rule{\hsize}{0.5pt}\\*[-2ex]

\textbf{Input:} The discriminant $d_K$ and the modulus $\ff$ of $K$.  

\textbf{Output:} A complete system of conjugates for $\epsilon(\cc)$ over $K$.

\begin{enumerate}
\item Compute $W$ and check if $K_\ff=H(\zeta_W)$.
\begin{itemize}
\item If YES, print: \texttt{Use class invariants together with roots of unity}
and return $0$.
\item If NO, find a class $\cc\neq 1$ in $\mbox{Cl}_\ff$ such that
$e_\cc=1$ and $\tau_\cc=\tau_1$. Go to the next step.
\end{itemize}
\item Compute $[u_1, v_1]$. Go to the next step.
\item Construct $\widehat{\alpha}$ and compute $[u_\cc, v_\cc]$. Go to the next
step.
\item Compute the list $[Q_1,\ldots,Q_m]$ of all reduced binary quadratic forms
in
$\mbox{Cl}(d_K)$ which is in one-to-one correspondence with $\mbox{Gal}(H/K)$,
and the list $[\beta_1,\ldots ,\beta_n]$ of matrices from $\mc{W}_{M,\tau_1} /
\{\pm I_2\}$ whose restriction to $\kf$ is in one-to-one correspondence with
$\mbox{Gal}(\kf/H)$. Go to the next step.
\item Compute the lists $[u_{Q_1},\ldots,u_{Q_m}]$ and $[\tau_{Q_1}, \ldots,
\tau_{Q_m}]$. Go to the next step.
\item for $1\leq i\leq n$, for $1\leq j\leq m$ 
\begin{itemize}
 \item Compute $$\epsilon(i,j)=\frac{\phi([u_\cc,v_\cc] \cdot \beta_i \cdot
u_{Q_j},
\tau_{Q_j})}{\phi([u_1,v_1] \cdot \beta_i \cdot u_{Q_j}, \tau_{Q_j}) }.$$
\end{itemize}
\item Return the matrix $\epsilon$.
\end{enumerate}
 \rule{\hsize}{0.5pt}\\*[-2ex]

We implemented this algorithm in PARI/GP, see \cite{PARI2}, for constructing
class fields and comparing our results with the existing ones. We can summarize
the construction above with the following theorem:

\begin{theorem}\label{thm:main}
Let $\mathfrak{f}$ be an ideal of $\oo$, and $f$ be the smallest positive
integer in $\mathfrak{f}$. Suppose that $\kf\neq H(\zeta_W)$, and let further
$\cc\neq 1$ be an ideal class of principle ideals such that $\sigma_\cc$ acts
trivially on $\zeta_W$. Then there exists an algorithm which computes numbers
$u_\cc,v_\cc, u_1,v_1 \in \frac{1}{f}\Z$ such that
\[\epsilon(\cc) =
\frac{\phi(u_\cc,v_\cc,\tau_1)} {\phi(u_1,v_1,\tau_1)}\in K_{\mathfrak{f}}.\] 
Moreover, there exists another algorithm which computes the complete system of
conjugates of $\epsilon(\cc)$ over $K$.
 \end{theorem}

In all our experiments, the value $\epsilon(\cc)$ appears to be indeed a
generator of $\kf$ over $K$. Moreover, the proof of \cite[Lemma~3.3]{JKS11}
suggests also that our quotient is a generator of $\kf/K$. Hence, we have the
following conjecture.

\begin{conjecture}\label{conj:ec}
The elliptic unit $\epsilon(\cc)$ generates $\kf$ over $K$ if the conditions of
Theorem~\ref{thm:main} hold.
\end{conjecture}

The following corollary follows immediately from Algorithm~I.

\begin{corollary}
Let the conditions of Theorem \ref{thm:main} hold, and Conjecture \ref{conj:ec}
be true. Then there exists an algorithm which computes the minimal polynomial
$h(x)\in\oo[x]$ of the generator $\epsilon(\cc)$ of $\kf$ over $K$.
\end{corollary}

\begin{remark}
Let $\cc$ be an ideal class modulo $\ff=(N)$, and let further $h_\cc(x)$
and $h_\phi(x)$ be the minimal polynomials of $\epsilon(\cc)$ and
$\phi(0,1/N,\tau_1)^{12N/(6,N)}$ of \cite{JKS11} over $\Q$, respectively.
Furthermore, suppose that $\gamma_\cc$ and $\gamma_\phi$ be the logarithm of
maximum of absolute values of the coefficients of $h_\cc(x)$ and $h_\phi(x)$,
respectively.

We compare heuristically the values $\gamma_\cc$ and $\gamma_\phi$. We expect
that the reduction factor can be measured by the exponent $$r(N) :=
\frac{\gamma_\phi}{\gamma_\cc} \approx \frac{12N}{\gcd(6,N)}$$ for arbitrarily
large $N$.

This reduction factor gives significantly better results compared with the
analogues results derived from the gonality estimates of modular curves and
their relation to the celebrated Selberg's eigenvalue conjecture for class
invariants, i.e. generators of ring class fields. In that case, the reduction
factor is a constant bounded by $96$ conjecturally and by $100.82$ provably, see
\cite[p. 25]{BrSt08}, whereas our method yields a reduction factor depending
linearly on $N$ which turns out to be very efficient especially for the cases of
large conductor and large discriminant.
\end{remark}

\section{Examples}

\begin{example}\label{nontrivial}
Let $K=\Q(\sqrt{-91})$ and $\ff=(5)$. The form class group of the discriminant
$d_K=-91$ is given by 
\[\mbox{Cl}(d_K) = \{[1,1,23],[5,3,5]\}.\]
Moreover, we find that
\[ \mc{W}_{5,\tau_1}/\{\pm I_2\} = \left\{
\begin{array}{l}
\sm{1}{0}{0}{1}, \sm{2}{0}{0}{2}, \sm{-1}{-23}{1}{0}, \sm{0}{-23}{1}{1},
\sm{2}{-23}{1}{3}, \vspace*{5pt} \\  \sm{-2}{-46}{2}{0}, \sm{-1}{-46}{2}{1},
\sm{0}{-46}{2}{2}
\end{array}
\right\}. \]
Giving these two sets is equivalent to enumerate all ideal classes in the ray
class group by Theorem~\ref{thm:jung}. We may choose $[u_1,v_1]=[0,1/5]$. The
function $\phi(a/5,b/5,z)$ is modular of level $M = 12 \cdot 5^2$ provided that
$(5,a,b)=1$. The number of roots of unity in the ray class field $K_{(5)}$ is
given by $W=10$. It follows that $\ell = 12\cdot5/10=6$.

\textbf{Part 1:} Let $\cc_1$ be the ideal class corresponding to the pair
$$(\sm{-1}{-46}{2}{1} , [5,3,5]).$$ We find that $u_Q \equiv
\sm{293}{169}{276}{49} \bmod{M}$, where $Q=[5,3,5]$. The matrix $\alpha =
\sm{9}{-46}{2}{11}$ is an element  of $\mc{W}_{M,\tau_1}$ congruent to
$\sm{-1}{-46}{2}{1}$ modulo $5$, and the determinant of $\alpha \cdot u_Q$ is
congruent to $1$ modulo $\ell$. We compute that $\det(\alpha \cdot u_Q)\equiv 3
\bmod{W}$, and as a result obtain $e_{\cc_1}=5$. Now the calculation of
$[0,1/5] \cdot \alpha \cdot
u_Q$ gives us 
\[ \epsilon(\cc_1)=\left[\frac{\phi\left(\frac{3622}{5},\frac{877}{5},\tau_Q
\right)\cdot\zeta_Q}{\phi\left(0,\frac{1}{5},\tau_1 \right)}\right]^5
\in K_{(5)}.\]
Here $\zeta_Q$ is a $12$-th root of unity depending only on the matrix $u_Q$.
Its precise value can be found by the action of the matrix $u_Q$
on Siegel $\phi$-function $\phi(u,v,z)$. It requires computing a decomposition
of $u_Q$ in terms of $S,T$ and $\sm{1}{0}{0}{d}$. This computation can be
eliminated if we choose $\cc$ as in the following part.

\textbf{Part 2:} Now let $\cc_2$ be the ideal class corresponding to the
pair $$\left( \sm{-1}{-46}{2}{1} , [1,1,23] \right).$$ We find that $u_Q \equiv
\sm{1}{0}{0}{1} \bmod{M}$. Observe that $\alpha=\sm{-1}{-46}{2}{1}$ has
determinant congruent to $1$ modulo $\ell$. We find that $\det(\alpha \cdot
u_Q)=1$ and therefore $e_{\cc_2}=1$. Now the calculation of $[0,1/5] \cdot
\alpha \cdot u_Q$ gives us
\[ \epsilon(\cc_2)=\frac{\phi\left(\frac{2}{5},\frac{1}{5},\tau_1
\right)}{\phi\left(0,\frac{1}{5},\tau_1 \right)} \in K_{(5)}.\]

Let $h_1(x)$ and $h_2(x)$ be the minimal polynomials of $\epsilon(\cc_1)$ and
$\epsilon(\cc_2)$ over $\Q$, respectively. The absolute values of coefficients
of $h_1(x)$ is much larger than the absolute values of coefficients of $h_2(x)$
due to the presence of the exponent $5$. As a comparison we give the largest
values $c_1$ and $c_2$ of the absolute value of the coefficients of polynomials
$h_1(x)$ and $h_2(x)$, respectively:
$$c_1=14039306026984320878929721009202946,$$
$$c_2=910425.$$
As expected, we find that $\log(c_1)/\log(c_2) \approx 5.73015$.
\end{example}

\begin{example}
Let $\ff$ be a product of degree one prime ideals such that $(\ff,
6\bar{\ff})=1$. There exists an integer $t_1$ with $\tau_1 \equiv t_1
\bmod{\bar{\ff}}$ by Chinese remainder theorem. We have $\bar{\ff} =
(f,\tau_1+t_1)$. Consider the decomposition of the principal ideal $(\tau_1+t_1)
= \bar{\ff}\mf{a}$ for some $\mf{a}\in\cc_0$. Choosing $\mathfrak{b}=\bar{\ff}$,
we find that $E(\cc_0) = \phi( 1/f,t_1/f,\tau_1)^{12f}\in\kf$. Thus we can
choose $[u_1,v_1] = [1/f,t_1/f]$.

In this case, we have $\kf \cap K_{\bar{\ff}}=H$, and it follows that all roots
of unity in $\kf$ lie in $H$. Let $\cc\neq 1$ be any ideal class in
$\mbox{Cl}_\ff$ consisting of principal ideals. The action of $\sigma_\cc$ is
trivial on the roots of unity of $\kf$, since $\zeta_W\in H$. Set $M=12f^2$. For
a given $\sigma \in \mbox{Gal}(\kf/H)$, we can find a matrix $\alpha_a =
\sm{a}{0}{0}{a} \in \mc{W}_{M,\tau_1}$ whose restriction to $\kf$ corresponds to
$\sigma$. Recall that $\ell=12f/W$. Without loss of generality we assume that
$\det(\alpha_a)=a^2\equiv 1 \pmod{\ell}$. Then any conjugate of $\epsilon(\cc)$
can be found by using the formula
\begin{align*}
 \epsilon(\cc)^{(\beta,Q)} & = \frac{\phi( [\frac{1}{f},\frac{t_1}{f}]
\cdot \alpha_a \cdot \alpha_b \cdot u_Q,\tau_Q)}{\phi([\frac{1}{f},\frac{
t_1}{f}] \cdot\alpha_b \cdot u_Q,\tau_Q)} = \frac{\phi(
[\frac{ab}{f},\frac{abt_1}{f}] \cdot
u_Q,\tau_Q)}{\phi([\frac{b}{f},\frac{ bt_1}{f}] \cdot u_Q,\tau_Q)}
\end{align*}
where $(\alpha_b,Q)$ is any pair with $b\in (\Z/M\Z)^*$ and
$Q\in\mbox{Cl}(d_K)$. 
\end{example}

\begin{example}
In this example we compare the result of Algorithm~I with the example given in
\cite[Example~3.8]{JKS11}. Let $K=\Q(\sqrt{-10})$ and let $\ff=(6)$.
It turns out that $\kf$ has 24 roots of unity. We find that
\[ \mc{W}_{6,\tau_1}/\{\pm I_2\} = \left\{
\begin{array}{l}
\sm{1}{0}{0}{1}, \sm{1}{-10}{1}{1}, \sm{3}{-10}{1}{3}, \sm{5}{-10}{1}{5},
 \vspace*{5pt} \\ \sm{1}{-20}{2}{1}, \sm{3}{-20}{2}{3}, \sm{5}{-20}{2}{5},
\sm{1}{-30}{3}{1}
\end{array} \right\}. \]
Among the $8$ principal ideal classes only those two corresponding to $1$
and $2\sqrt{-10}+3$ have matrices with determinant congruent to $1$ modulo $24$.
We compute the value $\ell=12\cdot f/24=3$. Thus we shall use
$\alpha=\sm{3}{-20}{2}{3}\in \mc{W}_{M,\tau_1}$, where $M=12\cdot 6^2$. The
matrix $\alpha$ has determinant congruent to $1$ modulo $\ell$. It follows
that
\[ \epsilon(\cc)=\frac{ \phi(\frac{2}{6}, \frac{3}{6}, \tau_1)}{\phi(0,
\frac {1}{6}, \tau_1) } \]
is an element of $K_{(6)}$. Multiplying $\epsilon(\cc)$ with $\zeta_{12}^5 \in
K_{(6)}$, a real element with the minimal polynomial
\[ {\min}(\epsilon(\cc)\zeta_{12}^5,K)= \begin{array}{l}
x^{16} + 8x^{15} - 18x^{14} - 68x^{13} + 50x^{12} \\
+ 108x^{11} - 44x^{10} - 28x^9 + 63x^8  \\
- 28x^7 - 44x^6 + 108x^5 + 50x^4 - 68x^3 \\
- 18x^2 + 8x + 1\end{array}\]
can be found. This has much smaller coefficients than the following minimal
polynomial computed in \cite{JKS11}:
\begin{align*}
 \min\left(\phi^{12}\left(0,\frac{1}{6},\tau_1\right),K\right) =
\begin{array}{l}
x^{16} + 20560x^{15} - 1252488x^{14} - 829016560x^{13}\\ 
- 8751987701092x^{12} + 217535583987600x^{11} \\
+ 181262520621110344x^{10} + 43806873084101200x^9 \\
- 278616280004972730x^8 + 139245187265282800x^7 \\
- 8883048242697656x^6 + 352945014869040x^5 \\
+ 23618989732508x^4 - 1848032773840x^3 \\
+ 49965941112x^2 - 425670800x + 1.\end{array}
\end{align*}
Similar to the previous example, we compare the largest values $c_1$ and $c_2$
of the absolute values of the coefficients of polynomials $h_1(x)$ and $h_2(x)$,
respectively:
$$c_1=278616280004972730,$$
$$c_2=108.$$
As expected, we find that $\log(c_1)/\log(c_2) \approx 8.57913$.
\end{example}

\begin{example}
Let $K=\Q(\sqrt{-11})$ and $\ff=(9)$. We compare in this example the
result of Algorithm~I with the example given in \cite[Example~3]{BeSch01}.
Bettner and Schertz introduce an element $\Theta\in\kf$ with the following
minimal polynomial:
\begin{align*}
 \min\left(\Theta,K\right) & =
\begin{array}{l}
x^{18} + 9x^{17} + 36 x^{16} + (-8  \tau_1 + 91) x^{15} \\ 
+ (-78  \tau_1 + 150) x^{14} + (-294  \tau_1 + 45) x^{13} \\
+ (-492  \tau_1 - 479) x^{12} + (-120  \tau_1 - 1020) x^{11} \\
+ (816  \tau_1 - 327) x^{10} + (1068  \tau_1 + 1469) x^9 \\
+ (-18  \tau_1 + 1707) x^8 + (-882  \tau_1 - 357) x^7 \\
+ (-288  \tau_1 - 1523) x^6 + (516  \tau_1 - 345) x^5 \\
+ (390  \tau_1 + 540) x^4 + (2  \tau_1 + 219) x^3 \\
+ (-6  \tau_1 - 15) x^2  + (6  \tau_1 + 15) x + 1.
\end{array}
\end{align*}
On the other hand, we compute $\epsilon(\cc) = \phi(7/3,-2/9,\tau_1) /
\phi(0,1/9,\tau_1)$ with the minimal polynomial below:
\begin{align*}
 \min\left(\epsilon(\cc),K\right) & =
\begin{array}{l}
x^{18} + 3 x^{17} + (-6  \tau_1 + 3) x^{16} + (5  \tau_1 - 4) x^{15} \\
+ (-6  \tau_1 + 18) x^{14} + (3  \tau_1 - 3) x^{13}  \\
+ (-12  \tau_1 + 40) x^{12} + (-6  \tau_1 + 6) x^{11}  \\
+ (-15  \tau_1 + 63) x^{10} - 2 x^9 + (15  \tau_1 + 78) x^8   \\
+ (6  \tau_1 + 12) x^7 + (12  \tau_1 + 52) x^6 + (-3  \tau_1 - 6) x^5 \\
 + (6  \tau_1 + 24) x^4 + (-5  \tau_1 - 9) x^3 + (6  \tau_1 + 9) x^2 \\
+ 3 x + 1.
\end{array}
\end{align*}
Let $c_\Theta$ and $c_{\epsilon(\cc)}$ be the absolute values of the
coefficients with largest absolute value of $\min\left(\Theta,K\right)$ and
$\min\left(\epsilon(\cc),K\right)$, respectively. We find that $\log(c_\Theta)/
\log(c_{\epsilon(\cc)}) \approx 1.76212$. Note that this is not the only
advantage. For different modulus, it is possible to compute polynomials
with smaller coefficients using Algorithm~I than the polynomials in
\cite{BeSch01} heuristically. Furthermore, there is no restriction on the
underlying modulus $\ff$ whereas the method in \cite{BeSch01} is only applicable
in very restrictive cases, and rather complicated with possibly higher powers
when $N$ gets larger. 
\end{example}

\section{Possible Applications}

We list in this section possible further applications and generalizations of the
construction in Algorithm~I.  

First of all, it could be interesting to investigate whether the set of elliptic
units $\epsilon(\cc)$ constructed by Algorithm I together with the units
from Hilbert class field generate a subgroup of the unit group $U(\oo_{\kf})$
of finite index. 

Secondly we aim at investigating the result of Klebel, \cite{Kle96}, to find
power integral basis of class fields over Hilbert class field by means of
employing the elliptic units $\epsilon(\cc)$ in a forthcoming research project.
The existence of such a basis yields a solution of certain Diophantine
equations, see \cite{Gaa02}. 

In a forthcoming paper we plan to generalize Algorithm~I to the case of complex
multiplication by an arbitrary order $\mathcal{O}$, see \cite{omerben}.  

Lastly it could be interesting to investigate the elements of Hilbert class
field, or more generally ring class fields, coming from relative norms over $H$
of elliptic units $\epsilon(\cc)$ constructed by Algorithm~I. This could
possibly yield other source of class invariants. These invariants could be used
in the applications of CM-theory such as primality proving, group and pairing
based cryptography, see for instance \cite{AtMr93}, \cite{Mor07}, \cite{BSS99},
\cite{BSS05} or \cite{FST06}.

{\small
\def\refname{References}
\newcommand{\etalchar}[1]{$^{#1}$}

\end{document}